\documentclass[conference]{IEEEtran}
\IEEEoverridecommandlockouts
\usepackage{cite}
\usepackage{amsmath,amssymb,amsfonts}
\usepackage{algorithmic}
\usepackage{graphicx}
\usepackage{textcomp}
\usepackage{xcolor}
\def\BibTeX{{\rm B\kern-.05em{\sc i\kern-.025em b}\kern-.08em
    T\kern-.1667em\lower.7ex\hbox{E}\kern-.125emX}}
\usepackage{amsthm}
\DeclareMathOperator{\PE}    {PE}
\DeclareMathOperator{\MPE}    {MPE_G}
\DeclareMathOperator{\PEG}    {PE_G}
\DeclareMathOperator{\MMSPE}    {MMPE}

\newcommand{\X}{\mathbf{X}} 
\newcommand{\U}{\mathbf{U}} 
\newcommand{\R}{\mathbb{R}} 
\newcommand{\N}{\mathbb{N}}
\newcommand{\Z}{\mathbb{Z}}
\newcommand{\V}{\mathcal{V}} 
\newcommand{\E}{\mathcal{E}}

\newcommand{\Nb}{\mathcal{N}} 
\newcommand{\CG}{\mathcal{G}_\U} 
\newcommand{\A}{\mathbf{A}} 
\newcommand{\card}[1]{\lvert#1\rvert}
\newcommand{\set}[2]{\{ \, #1 \, | \, #2 \, \} } 
\newcommand{\G}{{G}}

\newtheorem{proposition}{Proposition}
\newtheorem{definition}{Definition}
\newtheorem{example}{Example}
\newcommand{\map}[3]{ #1 \colon #2 \longrightarrow #3}
\begin{document}

\title{Multivariate permutation entropy,\\
	a Cartesian graph product approach

\thanks{J.S. Fabila-Carrasco and J. Escudero were supported by the Leverhulme
	Trust via a Research Project Grant (RPG-2020-158).}
}

\author{\IEEEauthorblockN{John~Stewart~Fabila-Carrasco}
\IEEEauthorblockA{\textit{School of Engineering} \\
\textit{Institute for Digital Communications}	\\
\textit{University of Edinburgh}\\
Edinburgh, EH9 3FB, UK \\
John.Fabila@ed.ac.uk}
\and
\IEEEauthorblockN{Chao~Tan}
\IEEEauthorblockA{\textit{School of Electrical and} \\
\textit{Information Engineering} \\
\textit{Tianjin University}\\
Tianjin, 300072, China \\
tanchao@tju.edu.cn}
\and
\IEEEauthorblockN{Javier~Escudero}
\IEEEauthorblockA{\textit{School of Engineering} \\
	\textit{Institute for Digital Communications}	\\
	\textit{University of Edinburgh}\\
	Edinburgh, EH9 3FB, UK \\
javier.escudero@ed.ac.uk}
}

\maketitle

\begin{abstract}
Entropy metrics are nonlinear measures to quantify the complexity of time series. Among them, permutation entropy is a common metric due to its robustness and fast computation. Multivariate entropy metrics techniques are needed to analyse data consisting of more than one time series. To this end, we present a multivariate permutation entropy, $\MPE$, using a graph-based approach. 

Given a multivariate signal, the algorithm $\MPE$ involves two main steps: 1) we construct an underlying graph $\G$ as the Cartesian product of two graphs $G_1$ and $G_2$, where $G_1$ preserves temporal information of each times series together with $G_2$ that models the relations between different channels, and 2) we consider the multivariate signal as samples defined on the regular graph $\G$ and apply the recently introduced permutation entropy for graphs. 

Our graph-based approach gives the flexibility to consider diverse types of cross channel relationships and signals, and it overcomes with the limitations of current multivariate permutation entropy.
\end{abstract}


\section{Introduction}

Entropy measurements are a common tool used in the analysis of time series to describe the probability distribution of the states of a system. Based on this concept, the seminal paper~\cite{Bandt2002} introduced the so-called permutation entropy ($\PE$) as a measure to quantify complexity in time series, a fundamental challenge in data analysis. This entropy involves calculating permutation patterns, i.e., permutations defined by comparing neighbouring values of the time series. $\PE$ has been applied in a wide range of fields: physical systems~\cite{Yan2012}, economics~\cite{Zunino2009}, and biomedicine~\cite{Cao2004, Olofsen2008}, among many other applications. 

There have been studies on the properties of permutation entropy, including extensions to higher regular domains~\cite{Morel2021} and irregular domains or graphs~\cite{Fabila2021}. Some modifications of $\PE$ consider nonlinear mappings to deal with the differences between the amplitude values~\cite{Azami2018,Rostaghi2016}, or weights in permutation patterns~\cite{Mitiche2017}. Previous research also has extended $\PE$ to different scales~\cite{Costa2002, Azami2016}, or studied its  dependencies with respect to random signals~\cite{Davalos2018}, autoregressive processes~\cite{Davalos2019} or high-order autoregressive processes~\cite{Davalos2020}.

Most biomedical and physical systems are multivariate. Therefore, univariate entropy metrics have been generalised to a multivariate setting, including: multivariate sample entropy~\cite{Ahmed2011}, multivariate dispersion entropy~\cite{Azami2019}, among others. A multivariate multiscale permutation entropy ($\MMSPE$) to analyse physiological signals is proposed in~\cite{Morabito2012a}. However, such algorithm extracts the permutation patterns from each channel  separately regardless of their cross-channel information. $\MMSPE$ 
treated multichannel signals as a unique block and without interactions between the channels. Thus, it works appropriately when the components of a multivariate signal are statistically independent but does not consider the spatial domain of time series.
\subsubsection*{Contributions}
We introduce a multivariate permutation entropy based on the Cartesian product of graphs. Such approach would enable us to, first, overcome the limitation of current multivariate permutation entropy and, second, give flexibility to consider diverse types of cross channel relationships and signals. 
\subsubsection*{Structure of the manuscript}
The outline of the paper is as follows: Section \ref{background} introduces the Cartesian product of graphs and the permutation entropy: univariate, multivariate and for graph signals. Section~\ref{graph-construction} presents the graph associated with a multivariate signal, and Section~\ref{method} presents the multivariate permutation entropy. Section~\ref{exp} shows how $\MPE$ applies to synthetic signals. The conclusions and future lines of research are presented in Section~\ref{conc}. 

\section{Background: graph and permutation entropies}\label{background}
In this section, we introduce general background information, including the definition of a graph and the Cartesian product (Section~\ref{sub:GraphProd}), the original permutation entropy (Section~\ref{sub:originalPE}) and the recently introduced permutation entropy for graph signals (Section~\ref{sub:PEG}).

\subsection{Graphs, Cartesian product and graph signal}\label{sub:GraphProd}
An \emph{undirected graph} (or simply graph) $G$ is defined as the pair $G = (\V,\E)$
which consists of a finite set of vertices or nodes $\V=\{1,2,3,\dots, n\}$, an edge set $\E \subset \{(i,j): i,j\in\V\}$.  The adjacency matrix $\A$ is the corresponding $N \times N$ symmetric matrix on edges with entries $1 = \A_{i j}= \A_{j i}$ if $(i,j)\in \E$ and $0$ otherwise.

A \emph{directed graph} or \emph{digraph} is a graph where each edge has an orientation or direction. 

The \emph{Cartesian product} of two graphs $\G= (\V,\E)$ and $\G'=(\V',\E')$, denoted $\G\square \G'$, is the graph defined by:
\begin{enumerate}
	\item the vertex set is given by: \[\V(\G\square \G')=\V\times \V'=\set{(v,v')}{v\in\V \text{ and } v'\in\V'}\:;\]
	\item two vertices $(v,v')$ and $(u,u')$ are adjacent in $\G\square \G'$ if and only if either
	\begin{itemize}
		\item $v=u$ and $v'$ is adjacent to $u'$ in $\V'$, or
		\item $v'=u'$ and $v$ is adjacent to $u$ in $\V$. 
	\end{itemize}
\end{enumerate}

A \emph{graph signal} is a real function defined on the vertices, i.e.,
$\map{\X}{\V}{\R}$. The graph signal $\X$ can be represented as an $n$-dimensional column vector.

\subsection{Original permutation entropy: univariate and multivariate}\label{sub:originalPE}
\subsubsection*{Univariate permutation entropy} For a time series $\textbf{X}=\left\{x_i\right\}_{i=1}^{n}$, the algorithm to compute $\PE$ is the following \cite{Bandt2002}:

	$1)$ For $2\leq m\in\N$ the \emph{embedding dimension} and $L\in\N$ the \emph{delay time}, the \emph{embedding vector} $\textbf{x}_i^m(L)\in\R^m$ is given by
$\textbf{x}_i^m(L)=\left( x_{i+jL}\right)_{j=0}^{m-1}=\left(x_i,x_{i+L},\dots, x_{i+(m-1)L}\right)$ 
for all $1\leq i \leq n-(m-1)L$.

$2)$ The embedding vector $\textbf{x}_i^m(L)=\left( x_i,x_{i+L},\dots, x_{i+(m-1)L}\right)$ is arranged in the increasing order vector: $\left( x_{i+(k_1-1)L}\leq x_{i+(k_2-1)L} \leq \dots \leq  x_{i+(k_m-1)L}\right)$.   We use the convention in~\cite{Cao2004}. In the case of equal values, the order is given by the corresponding $k's$. Therefore, any embedding vector $\textbf{x}_i^m(L)$ is uniquely mapped onto the vector $(k_1,k_2,\dots,k_m)\in \N^m$.

$3)$ The relative frequency for the distinct permutation $\pi_1,\pi_2,\dots,\pi_k$, where $k=m!$, is denoted by $p(\pi_1),p(\pi_2),\dots,p(\pi_k)$. The permutation entropy $\PE$ for the time series $\textbf{X}$ is computed as the normalised Shannon entropy 
for the $k$ distinct permutations as follows
\[
\PE=-\dfrac{1}{\ln(m!)}\sum_{i=1}^{m!} p(\pi_i) \ln p(\pi_i)\; .
\] 
\subsubsection*{Multivariate permutation entropy} $\MMSPE$ is proposed in \cite{Morabito2012a}. Let $\U$ be a multivariate signal, $\MMSPE$ applies steps $1)$ and $2)$ from the original $\PE$ for each channel. The difference is step $3)$, where the probability distribution aggregates the frequency of patterns from all channels in the multivariate signal, but it does not account for inter-channel relationships, i.e. 

$3)$ The relative frequencies are denoted by $\{\pi_{i,j}\}$, then the marginal relative frequencies describing the distributions of the patterns is defined by: $P_{j}=\sum \pi_{i,j}$ for $j=1,2,\dots,m!$. The multivariate $\MMSPE$ is computed as the normalised Shannon entropy 
for the marginal relative frequencies:
\[
\MMSPE=-\dfrac{1}{\ln(m!)}\sum_{j=1}^{m!} P_j \ln P_j\; .
\] 
\subsection{Permutation entropy for graph signals }\label{sub:PEG}
Let $G = (\V,\E)$ be a graph, $\A$ its adjacency matrix and $\textbf{X}=\left\{x_i\right\}_{i=1}^{n}$ be a signal on the graph. The permutation entropy for the graph signals $\PEG$ is defined in~\cite{Fabila2021} as follows: 

$1)$ For $2\leq m\in\N$ the \emph{embedding dimension}, $L\in\N$ the \emph{delay time} and for all $i=1,2,\dots,n$ we define 
  $y_{i}^{kL}= \frac 1 {\card{\Nb_{kL}(i)}} \sum_{j \in \Nb_{kL}(i)} x_j=\frac{1}{\card{\Nb_{kL}(i)}}(\A^{kL}\X)_i\;$, where
  	$\Nb_k(i)=\set{j\in \V}{\scalebox{.9}[1.0]{it exists a walk on $k$ edges joining $i$ and $j$}}\;$. Hence, we construct the embedding vector $\textbf{y}_i^{m,L}\in\R^m$ given by $\textbf{y}_i^{m,L}=\left( y_{i}^{kL}\right)_{k=0}^{m-1}=\left(y_i^0,y_{i}^{L},\dots y_{i}^{(m-1)L}\right)\;$. 
  	
$2)$ The embedding vector $\textbf{y}_i^{m,L}$ is arranged in increasing order.

$3)$  The relative frequency for the distinct permutation $\pi_1,\pi_2,\dots,\pi_k$, where $k= m!$, is denoted by $p(\pi_1),p(\pi_2),\dots,p(\pi_k)$. The permutation entropy $\PEG$ for the graph signal $\textbf{X}$ is computed as the normalised Shannon entropy
$$
\PEG=-\dfrac{1}{\ln(m!)}\sum_{i=1}^{m!} p(\pi_i) \ln p(\pi_i)\;.
$$
For time series, $\PEG$ reduces to $\PE$. In particular, if $\textbf{X}$ is a time series and $G$ the directed path on $n$ vertices, then for all $m$ and $L$, the equality holds:
$\PE(m,L)=\PE_{G}(m,L)$ (see~\cite[Prop.~3]{Fabila2021}).	
\section{Construction of the graph} \label{graph-construction}
In this section, we will associate a graph for each multivariate signal. Let $\U$ be a multivariate signal; we will construct a 2D graph (using the Cartesian product). One dimension will preserve the temporal information, and another will preserve the cross-channel information.
\subsection*{Dimension 1. Temporal information} 
We associate the directed path with a time series, where a vertex represents each sample time. A \emph{directed path} on $n$ vertices is a directed graph that joins a sequence of different vertices with all the edges in the same direction and is denoted by $\overrightarrow{P_n}$, i.e. its vertices are $\{1,2,\dots,n\}$ and its arcs $(i,i+1)$ for all $1\leq i \leq k-1$. An example is depicted in Fig.~\ref{fig1}(a). 
\subsection*{Dimension 2. Relationships between channels} $\U=\{U_s\}^{s=1,2,\dots,p}$ is a multivariate signal consisting on a set of $p$ time series (or channels). Let $I_p$ be the graph with $p$ vertices representing the interaction between different channels, i.e. $U_i$ and $U_j$ are adjacent in the graph $I_p$ if and only if they interact. 

If we do not have any a priori information about the interactions between channels, by default, we will consider equal interactions between all channels. Complete graphs represent such relations, i.e., we will set $I_p=K_p$ as the complete graph with $p$ vertices, see an example in Fig.~\ref{fig1}(b).

\begin{definition}\label{def:graph}
	Let $\U=\{u_{t,s}\}_{t=1,2,\dots,n}^{s=1,2,\dots,p}$ be a multivariate time series with $p-$channels of length $n$ and with $I_p$ the graph of interactions between channels. We define $\CG$ as the graph associated with $\U$ and given by:
	\[\CG:=\overrightarrow{P_n}\square I_p\;.\]
	We set $I_p$ as the complete graph in case of not having additional information about the interaction of the channels. 
\end{definition}
\begin{figure}[htbp]
	\centerline{\includegraphics[width=80mm]{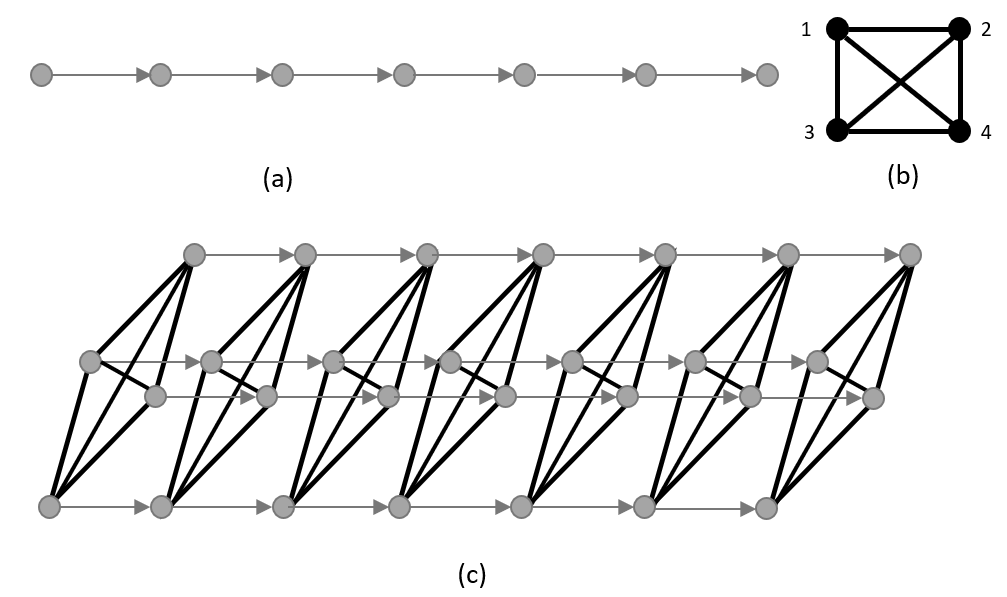}}
	\caption{(a) Directed path with seven vertices, denoted by $\protect\overrightarrow{P_7}$. (b) Interactions between the four channels are encoded with the complete graph on four vertices, denoted by $K_4$. (c) The Cartesian product $\protect\overrightarrow{P_7}\square K_4$.}
	\label{fig1}
\end{figure}

By construction, the vertices in $\CG$ and the sample points in the multivariate signal $\U$ are indexed by the same set, i.e.,	
\[{\V(\CG)=\V(\overrightarrow{P_n}\square I_p)=\V(\overrightarrow{P_n})\times \V(I_p)=\Z_n\times \Z_p}\;;\]
hence, we can consider $\U$ as a graph signal defined on the regular domain $\CG$.

The constructed $\CG$ is a 2D domain, time defines one dimension, and cross-channel dependencies define another one. Hence, the graph $\CG$ preserve the temporal/dependency structure of the multivariate signal $\U$. This is not the case for the currently available implementation of multivariate $\PE$.
\begin{example}
 Consider a multivariate time series with four channels and seven sample points, i.e., $\U=\{u_{t,s}\}_{t=1,2,\dots,7}^{s=1,2,3,4}$,  and we do not have any additional information between channels. By default, we will assume all channels interact with each other. Fig.~\ref{fig1}(c) shows the graph $\CG$ constructed in Def.~\ref{def:graph}. 
 
 An example of a multivariate signal where not all channels interact with each others is shown in Fig.~\ref{fig3}(a). The adjacency matrix associated with the graph $\CG$ is shown in Fig.~\ref{fig3}(b). The graph $\CG$ constructed according to Def.\ref{def:graph} is depicted in Fig.~\ref{fig3}(c).
\begin{figure}[htbp]
	\centerline{\includegraphics[width=80mm]{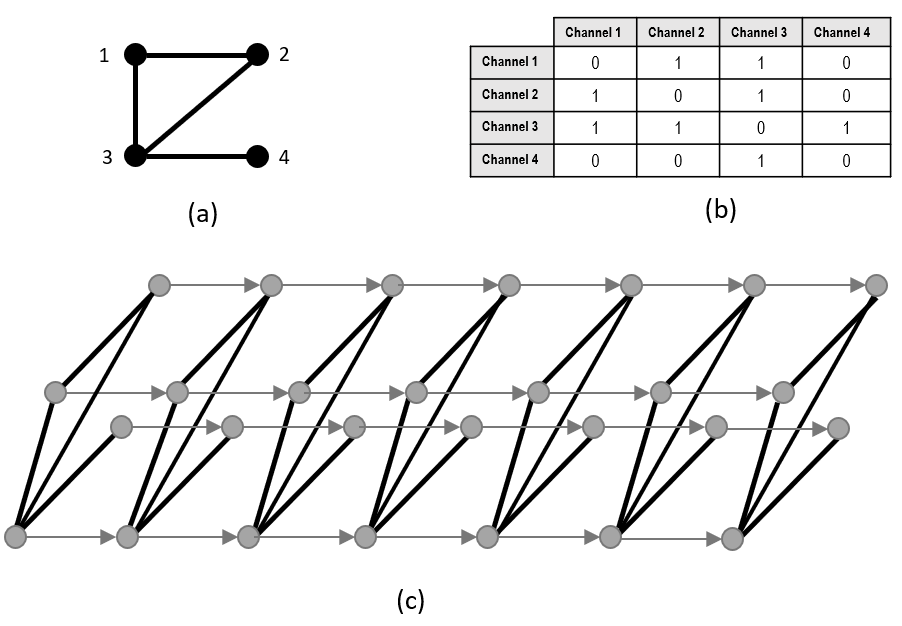}}
	\caption{(a) An undirected graph $\G$ represents the relationships between channels. (b) The adjacency matrix corresponding to the graph $\G$. (c) The Cartesian product $\protect\overrightarrow{P_7}\square G$.}
	\label{fig3}
\end{figure}
\end{example}

\subsubsection*{More complex relations between channels} Observe that the only imposed condition on the graph $I_p$ is the number of vertices, i.e., $I_p$ has as many vertices as the number of channels. Non-complete graphs can model other dependencies between channels. 

We will use undirected edges for bidirectional relationships between channels. We also can use directed edges for unidirectional interaction and include weighted edges for heterogeneous relations. Hence, in general, $I_p$ would be a weighted (directed or undirected) graph.

\subsubsection*{Regular structure but not periodic structure} Covering graphs or periodic graphs are used as models of chemical compounds, like graphene nanoribbons~\cite{Fabila2019}. The graph $\CG$ is not a covering graph, but it can be considered as a geometrical perturbation (see, e.g.,~\cite{Fabila2020}) of a periodic graph. Hence, some properties of the periodic graphs can be preserved in $\CG$, including spectral properties~\cite{Fabila2018}. Such properties are important in graph signal processing~\cite{Ortega2018,Stankovic2019} and combinatorics~\cite{Fabila2021b}. This could be useful to formulate more general improved multivariate signals entropies. 

\section{Multivariate permutation entropy ($\MPE$)}\label{method}
In this section, we define the multivariate permutation entropy $\MPE$. We use the permutation entropy for graph signals $\PEG$ (Section~\ref{sub:PEG}) and the graph construction described in Def.~\ref{def:graph}.

\begin{definition}{\textbf{Multivariate permutation entropy ($\MPE$)}}\label{def}
	
	Let $\U=\{u_{t,s}\}_{t=1,2,\dots,n}^{s=1,2,\dots,p}$ be a multivariate time series with interaction graph $I_p$ between channels
	\begin{enumerate} 
		\item \textbf{Graph construction.} Construct the graph $\CG$ described in Def.~\ref{def:graph}, i.e.,
			\[\CG:=\overrightarrow{P_n}\square I_p\;.\]
		\item \textbf{Graph signal.} Consider $\U$ as a signal defined on the graph $\CG$, i.e.,
		\[\map{\U}{\V(\CG)}{\R}\;.\]
		\item \textbf{$\mathbf{PE}$ for graph signals.} The \emph{multivariate permutation entropy ($\MPE$)} is defined as the permutation entropy for the graph signal $\PEG$ (see Section~\ref{sub:PEG}) for the signal $\U$ and the graph $\CG$, i.e., 
		\begin{equation*}
			\MPE=\PEG(\U)\;.
		\end{equation*} 
	\end{enumerate} 
\end{definition}

Proposition~\ref{prop} proves some important relations between $\MPE$ and $\PE$ metrics presented in the literature (see Table~\ref{tab1}). 
\begin{proposition}\label{prop}
	Let $\U=\{u_{t,s}\}_{t=1,2,\dots,n}^{s=1,2,\dots,p}$ be a multivariate time series with interaction graph $I_p$ between channels:
	\begin{enumerate}
		\item If $p=1$, then $\MPE(\U)=\PE$.
		\item If $s=1$, then $\MPE(\U)=\PEG$.
		\item If $I_p$ is the graph defined by $p$ isolated vertices, then \[\MPE(\U)=\MMSPE\:.\]
		\item If $I_p$ is a directed path on $p$ vertices, then \[\MPE(\U)\approx\PE_{2D}\:.\]
		\item If $\U=\{U_t\}_{t=1,2,\dots,n}^{s=1,2,\dots,p}$, then $\MPE(\U)\approx\PE(U_t)$.
	\end{enumerate}
\end{proposition}

\begin{proof}
	$1)-3)$ are easy properties and follow from the entropy definitions (Table~\ref{tab1}).\\
	$4)$  If $I_p$ is a directed path on $p$ vertices, then $\CG$ is a directed grid graph with $n\times p$ vertices, i.e., $\CG=\overrightarrow{P_n}\square \overrightarrow{P_p}$. Hence, by definition of $\MPE$, the algorithm reduces to apply $\PEG$ to the signal $\U$ defined on the grid $\CG$, and the performance of $\PEG$ and $\PE_{2D}$ are similar (see~\cite[Sec.~IV]{Fabila2021}).\\
	$5)$ We prove the case $m=2$ and $L=1$; the other cases are analogous. By definition of $\PEG$, each vertex belongs to the permutation pattern $\pi_1$ or $\pi_2$. For every vertex $v_i\in\V(\overrightarrow{P_n})$, where $v_i$ is not the last vertex of the path,  it is easy to show that the set of vertices $\set{(v_i,v_j)\in\V(\CG)}{j=1,2,\dots,p}$ also belong to $\pi_1$ (similarly with $\pi_2$). Hence, the relative frequencies are preserved in $\CG$ and $\overrightarrow{P_n}$ (except for at most $s$ vertices corresponding to the last vertex of each path). Then, for a large $n$, the values of its corresponding Shannon entropies are close enough.   
 \end{proof}
\begin{table}[htbp]
\caption{Summary of some permutation entropy metrics.}
\begin{center}
\begin{tabular}{|c|c|c|}
\hline
\textbf{Entropy metric}  & \textbf{Properties/Limitations}\\
\hline
\shortstack{ $\mathbf{PE}$\\ Permutation Entropy \\ Reference:~\cite{Bandt2002} -\cite{Olofsen2008} }& \shortstack{Analyse univariate time series \\ Simple and computationally fast \\ Multiscale extension}\\
 \hline
 \shortstack{ $\mathbf{PE_{2D}}$: Regularly \\sampled 2D data $\mathbf{PE}$ \\ Reference:~\cite{Morel2021} }& \shortstack{Analyse bidimensional data\\ Multiscale extension \\ Valuable for texture analysis}\\
 \hline
\shortstack{ $\mathbf{PE_G}$ \\ $\mathbf{PE}$ for Graph Signals\\ Reference:~\cite{Fabila2021} }&     \shortstack{Analyse graph signals \\(including: time series and image) \\ No multiscale extension yet}\\
\hline
\shortstack{ $\mathbf{MPE}$ \\Multivariate Multiscale $\mathbf{PE}$\\ Reference:~\cite{Morabito2012a}  }&  \shortstack{Analyse multivariate data but\\as a unique block (no interactions)\\Multiscale included}\\

\hline
\shortstack{$\mathbf{MPE_G}$ \\Multivariate $\mathbf{PE}$,\\ a graph product approach\\ Reference:~Definition~\ref{def}}& \shortstack{Analyse multivariate data\\including cross channel relationships\\Use $\mathbf{PE_G}$ for a graph Cartesian product\\No multiscale extension yet}\\  
\hline
\end{tabular}
\label{tab1}
\end{center}
\end{table} 

\section{Experiments}\label{exp}
In this section, we apply the algorithm to a set of multivariate synthetic signals used in the study of dynamical systems.   
\subsection{The H\'enon map}
In discrete-time dynamical systems, one of the most studied is the H\'enon map introduced in \cite{Henon}. Using $\MPE$ we can detect dynamical changes in the two-dimensional system defined by the equations:
\begin{align*}
	x_{n+1}&=1-ax_{n}^{2}+y_{n}\\
	y_{n+1}&=bx_{n}\;.
\end{align*}

The map depends on two parameters: $a$ and $b$. For the values $a=1.4$ and $b=0.3$ indicate the existence of a strange attractor; hence the map is chaotic. With $b=0.3$ and for other values of the parameter $a$, Fig.~\ref{fig10} shows that the map may be periodic, chaotic or intermittent.
\begin{figure}[htbp]
	\centerline{\includegraphics[width=81mm]{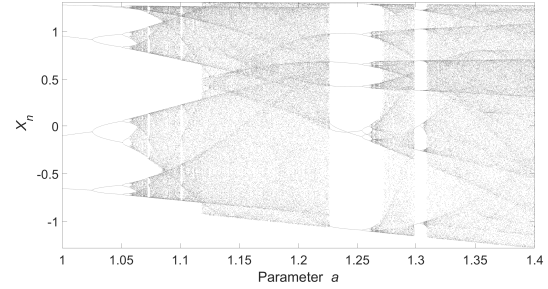}}
	\caption{Orbit diagram for the H\'enon map with $b=0.3$.}
	\label{fig9}
\end{figure}
We will analyse the parameter $a$, we consider $a\in[1,1.4]$ with increments in steps of $.0001$. For each iteration, we define the multivariate signal $\U=\{x_t,y_t\}_{t=1,2,\dots,n}$. The initial condition considered are $x_1= 0.5$, $y_1=0.1$ with $n=100$ (similar results are obtained for other values). We apply $\MPE(\U)$ for detecting the dynamic between the two signals. We also apply $\PE$ for each univariate signal and $\MMSPE$ to the multivariate signal, results obtained by $\PE$ and $\MMSPE$ are similar. Fig.~\ref{fig10} shows the entropy values for $\MPE(\U)$ and $\PE$ for $m=3$ and $L=1$.  
\begin{figure}[htbp]
	\centerline{\includegraphics[width=70mm]{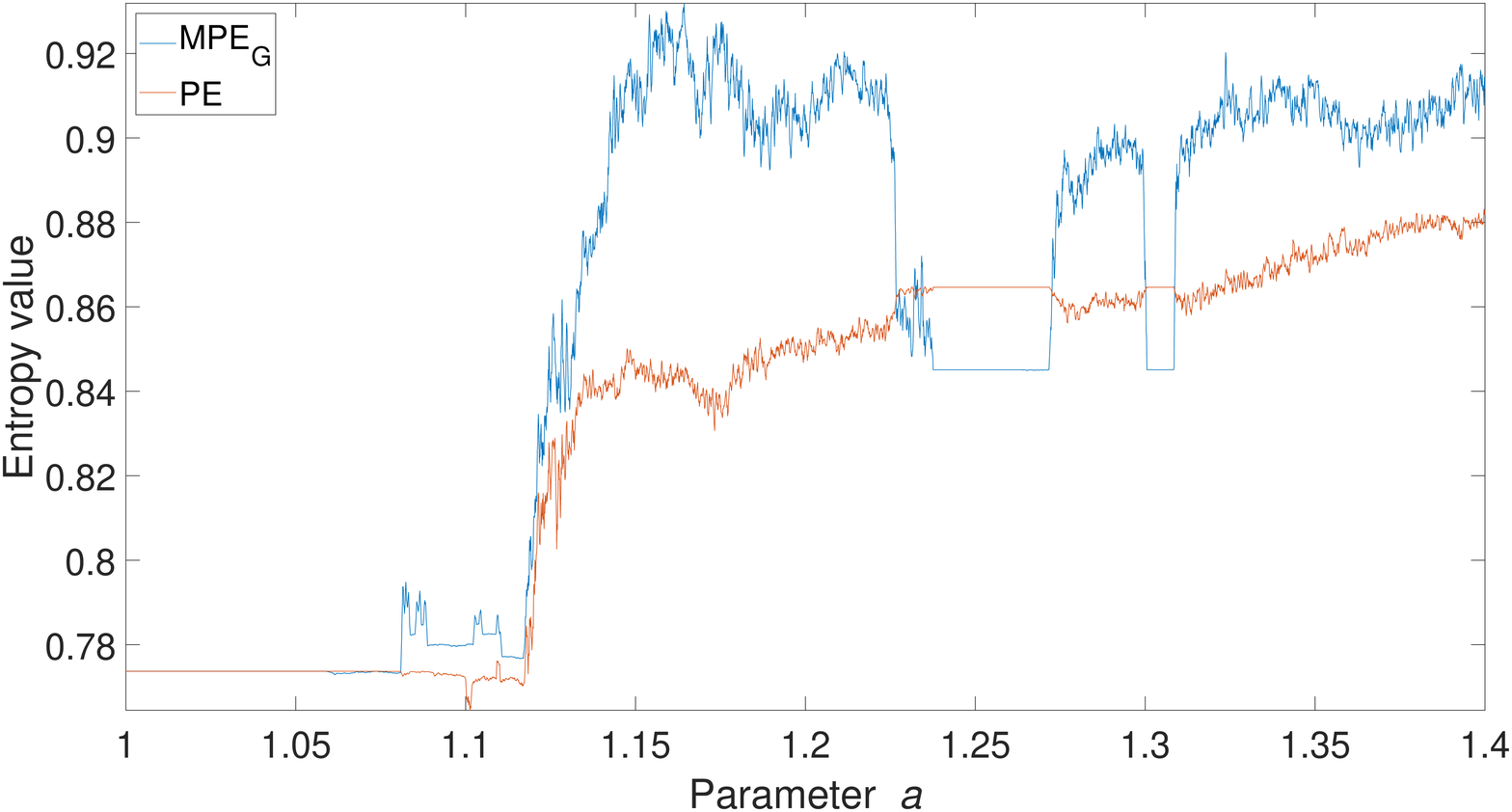}}	
	\caption{Entropy values computed for $m=3$ and  $L=1$.}
	\label{fig10}
\end{figure}

Our algorithm is able to detect chaotic behaviour and windows of stability. Moreover, the wider gap between values of $\MPE$ indicates a larger sensitivity.

\subsection{Lorenz system}
The Lorenz system is an example of a system of ordinary differential equations. This system has important applications in mechanics, biology, and circuit theory~\cite{Lorenz}; and is given by three simultaneous equations:
\begin{align*}{x'}&=\sigma (y-x),\\{y'}&=x(\rho -z)-y,\\{z'}&=xy-\beta z.\end{align*}
Lorenz used the values $\sigma =10$ and $\beta =8/3$. It is well-known that for $\rho =28$, the system showed chaotic behaviour. For $\rho<1$, the origin is a global attractor, i.e., all orbits converge to a unique equilibrium point \cite{Lorenz}. $\MPE$ algorithm captures this fact. Table~\ref{tab2} shows the entropy values computed $\MPE$ for $L=1$ and $m=3,4,5,6,7$. Entropy values are larger when $\rho>1$, reflecting more complexity, while for $\rho<1$, the system tends to the equilibrium; hence, lower values are obtained.
\begin{table}[htbp]
	\caption{Entropy values for the Lorenz system.}
	\begin{center}
		\begin{tabular}{|c|c|c|c|c|c|}
			\hline	
			&$m=3$&$m=4$&$m=5$&$m=6$&$m=7$\\
						\hline		
			$\rho=0.8$ &  0.4524   & 0.2860   & 0.1981   & 0.1477  &  0.1166\\
						\hline	
			$\rho=0.9$ &  0.4538   & 0.2878   & 0.1986   & 0.1489  &  0.1169\\
						\hline	
			$\rho=1.2$ &  0.7258   & 0.6673   & 0.5564   & 0.4478 &   0.3787\\
						\hline	
			$\rho=1.3$ &  0.7226  &  0.6872   & 0.5905   & 0.4815 &   0.4136\\	
						\hline							
		\end{tabular}
		\label{tab2}
	\end{center}
\end{table} 
\section{Conclusions and future work}\label{conc}
 We introduced a multivariate permutation entropy to quantify the complexity of multivariate time series. The algorithm proposed use the Cartesian product of graphs and the recently introduced permutation entropy for graph signals~\cite{Fabila2021}. Our graph-based approach considers diverse type of cross channel relationships and overcomes with the limitations of current multivariate permutation entropy. 

Future lines of research-based on the present work are:

\subsubsection{Multivariate dispersion entropy} Using a similar graph-technique presented in this paper, some univariate metrics can be generalised to multivariate metrics, including Dispersion Entropy.

\subsubsection{Multiscale permutation entropy for graph signals} $\MPE$ requires $\PEG$ in its computation. Multiscale entropy for time series involves downsampling or a coarse-graining process. Such process is unclear for signals defined in graphs; hence a multiscale $\PEG$ is still an open issue.

\subsubsection{Interaction between channels changing with time} The graph $\CG$ (Def.~\ref{def:graph}) uses the Cartesian product. Implicitly, we assumed the relationships between channels are preserved along time. We will explore constructions involving changes in the interactions between channels respect to time. 

\subsubsection{Irregular domains and real-world data} The presented $\MPE$ deals with 2D constructions and synthetic signals. We will explore similar entropies techniques for irregular domains and apply to real-world data, including biomedical signals and phase-flow patterns.

\vspace{12pt}
\color{red}

\end{document}